\documentclass[oneside]{amsart}

\usepackage[letterpaper,body={12.3cm,18.5cm}, mag=1000]{geometry}
\usepackage{amssymb}
\usepackage{amsthm}
\usepackage{amscd}

\numberwithin{equation}{section}
\theoremstyle{plain}

\newtheorem{thm}{Theorem}[section]
 
 \newtheorem{lemma}[thm]{Lemma}
\newtheorem{prop}[thm]{Proposition}

\newtheorem*{thma}{Theorem A}
\newtheorem*{thmb}{Theorem B}

\theoremstyle{definition}

\newcommand{\dlabel}[1]{\ifmmode \text{\ttfamily \upshape [#1] } \else
{\ttfamily \upshape [#1] }\fi \label{#1}}

\newcommand{\C}{\operatorname{C} }

\newcommand{\Z}{\operatorname{Z} }

\newcommand{\het}{\operatorname{ht}}

\newcommand{\gen}[1]{\left < #1 \right >}
\newcommand{\Aut}{\operatorname{Aut} }

\newcommand{\Hom}{\operatorname{Hom} }
\newcommand{\Inn}{\operatorname{Inn} }

\newcommand{\Autcent}{\operatorname{Autcent} }

\begin{document}

\setlength{\parindent}{1em}
\setlength{\baselineskip}{15pt}

\title{On finite $p$-groups whose automorphisms are all central}

\author{Vivek K. Jain}

\author{Manoj K. Yadav}
\address{School of Mathematics, Harish-Chandra Research Institute, Chhatnag Road, Jhunsi, Allahabad 211019, INDIA.}
\email{vkj@hri.res.in; myadav@hri.res.in}


\begin{abstract}
An automorphism $\alpha$ of a group $G$ is said to be central if $\alpha$ commutes with every inner automorphism of $G$. We construct a family of non-special finite $p$-groups having abelian automorphism groups. These groups provide counterexamples to a conjecture of A. Mahalanobis [Israel J. Math., {\bf 165} (2008), 161 - 187].
We also construct a family of finite $p$-groups having non-abelian automorphism groups and all automorphisms central.  This solves a problem of I. Malinowska [Advances in group theory,  Aracne Editrice, Rome 2002, 111-127].
\end{abstract}

\maketitle

\section{Introduction}

Let $G$  be a finite group. An automorphism $\alpha$ of $G$ is called \emph{central} if $x^{-1}\alpha(x) \in \Z(G)$ for all $x \in G$, where $\Z(G)$ denotes the center of $G$. The set of all central automorphisms of $G$ is a normal subgroup of $\Aut(G)$, the group of all automorphisms of $G$. We denote this group by $\Autcent(G)$. Notice that $\Autcent(G) = \C_{\Aut(G)}(\Inn(G))$, the centralizer of $\Inn(G)$ in $\Aut(G)$, and $\Autcent(G) = \Aut(G)$ if $\Aut(G)$ is abelian.  We denote the commutator and Frattini subgroup of $G$ with $\gamma_2(G)$ and $\Phi(G)$, respectively.  Let  $G^p = \gen{x^p \mid x \in G}$ and $G_p = \gen{x \in G \mid x^p = 1}$.  For finite abelian groups $H$ and $K$, $\Hom(H, K)$ denotes the group of all homomorphisms from $H$ to $K$. Throughout the paper, $p$ always denotes an odd prime.

In this paper we construct examples of finite $p$-groups whose automorphisms are all central. First we consider the case when $\Aut(G)$ is abelian for a given group $G$.
In 1908, P. Hilton \cite[p 233]{pH08} asked the following question: Whether a non-abelian group can have an abelian group of isomorphisms (automorphisms). An affirmative answer to this question was given by G. A. Miller \cite{gM13} in 1913. He constructed a non-abelian group $G$ of order $64$ such that $\Aut(G)$ is abelian and has order $128$. More examples of such finite $2$-groups were constructed by R. R. Struik \cite{rS82} in 1982, M. J. Curran  \cite{mC87} in 1987 and  A. Jamali \cite{aJ02} in 2002. 
In 1974, H. Heineken and H. Liebeck \cite{HL74} showed that for any finite group $K$, there exists a finite $p$-group $G$ such that $\Aut(G)/\Autcent(G)$ is isomorphic to $K$. In particular, for $K = 1$, this provides  a $p$-group $G$ such that $\Aut(G) = \Autcent(G)$ is an elementary abelian group. In 1975, D. Jonah and M. Konvisser \cite{JK75} constructed $4$-generated groups of order $p^8$ such that $\Aut(G) = \Autcent(G)$ and $\Aut(G)$ is an elementary abelian group of order $p^{16}$, where $p$ is any prime. 

In 1927, C. Hopkins \cite{cH26} proved, among other things, that if $G$ is a group such that $\Aut(G)$ is abelian, then $G$ can not have a non-trivial abelian direct factor. But this result is not true for $2$-groups, as proved by B. Earnley in his thesis \cite[Theorem 2.3]{bE75} in 1975.  
Among other things, Earnley proved (i) there is no group $G$ of order $p^5$ such that $\Aut(G)$ is abelian, (ii) for each positive integer $n \ge 4$, there exist $n$-generated $p$-groups $G$ such that $\Aut(G)$ is abelian. On the way to constructing finite $p$-groups of class $2$ such that all normal subgroups of $G$ are characteristic, in 1979 H. Heineken \cite{hH79} produced groups $G$ such that $\Aut(G)$ is abelian. In 1994, M. Morigi \cite{mM94} proved that there exists no group of order $p^6$ whose group of  automorphisms is abelian and constructed groups $G$ of order $p^{n^2 + 3n +3}$ such that $\Aut(G)$ is abelian, where $n$ is a positive integer. In particular, for $n=1$, it provides a group of order $p^7$ having an abelian automorphism group. Further in 1995, M. Morigi \cite{mM95} proved that the minimal number of generators for a $p$-group with abelian automorphism group is $4$. In 1995, P. Hegarty \cite{pH95} proved that if $G$ is a non-abelian $p$-group such that $\Aut(G)$ is abelian, then $|\Aut(G)| \ge p^{12}$, and the minimum is obtained by the group of order $p^7$ constructed by M. Morigi. Moreover, in 1998 G. Ban and S. Yu \cite{BY98} obtained independently the same result and proved that if $G$ is a group of order $p^7$ such that  $\Aut(G)$ is abelian, then $|\Aut(G)| = p^{12}$ (the last result is true for all primes, not only for $p$ odd). 

We would like to remark here that all the examples mentioned above are special $p$-groups, where $p$ is  an odd prime. Until recently, no non-special $p$-group $G$ was known (to the best of our knowledge) such that $\Aut(G)$ is abelian. Our last statement is supported by the following conjecture of A. Mahalanobis \cite{aM08}:

\vspace{.2in}

\noindent{\bf Conjecture.} \emph{For an odd prime $p$, let $G$ be a finite $p$-group such that $\Aut(G)$ is abelian. Then $G$ is a special $p$-group.}
\vspace{.2in}

We construct a family of counterexamples to this conjecture in the following theorem, which we prove in Section 2.
\begin{thma}
 Let $m = n + 5$ and $p$ be an odd prime, where $n$ is a positive integer greater than or equal to $3$. Then there exists a $4$-generated group $G$ of order $p^m$ and exponent $p^n$ such that $\Aut(G)$ is abelian, but $G$ is not special. Moreover, $|\Aut(G)| = p^{n+10}$.
\end{thma}

Now we consider finite $p$-groups $G$ for which $\Aut(G) = \Autcent(G)$ is non-abelian. In 1982, M. J. Curran \cite{mC82} constructed groups $G$ of order $2^7$ such that $\Aut(G) = \Autcent(G)$ is non-abelian. Further, in 1984, J. J. Malone \cite{jM84} constructed $p$-groups for odd primes such that $\Aut(G) = \Autcent(G)$ is non-abelian. We would like to remark here that the groups of Curran and Malone have direct factors. More
precisely these groups were constructed by taking direct products of abelian (cyclic) $p$-groups and  groups $G$ such that $\Aut(G)$ is abelian. Examples of $2$-groups $G$ such that $G$ does not have an abelian direct factor and $\Aut(G) = \Autcent(G)$ is non-abelian were constructed by 
S. P. Glasby \cite{sG86} in 1986. Until recently, no examples of such $p$-groups were known (to the best of our knowledge) for an odd prime $p$.
Our last statement is supported by the following problem of I. Malinowska \cite[Problem 13]{iM02}.

\vspace{.2in}

\noindent{\bf Problem.} \emph{ For an odd prime $p$, find a $p$-group $G$ which has no non-trivial abelian direct
factor and $\Aut(G) = \Autcent(G)$ is non-abelian. }

\vspace{.2in}

We construct examples of such groups in the following theorem, which we prove in Section 3.

\begin{thmb}
 Let $m = n + 7$ and $p$ be an odd prime, where $n$ is a positive integer greater than or equal to $3$. Then there exists a group $G$ of order $p^m$, exponent $p^n$ and with no non-trivial abelian direct factor such that $\Aut(G) = \Autcent(G)$ is non-abelian. 
\end{thmb}

We would like to remark that before writing proofs of the above theorems, we used GAP \cite{gap}  to establish the validity of these results for the following pairs $(p, n)$: $\{(3, n) \mid 3 \le n \leq 7\}$, $\{(5,n) \mid 3 \le n \leq 7\}$, $\{(7,n) \mid 3 \le n \leq 5\}$.

\section{Groups $G$ with $\Aut(G)$ abelian}
In this section we construct an infinite family of finite $p$-groups $G$ such that $G$ is not a special $p$-group and $\Aut(G)$ is abelian, where $p$ is an odd prime. 

Let $n$ be a natural number greater than $2$ and $p$ an odd prime. Define
\begin{eqnarray}\label{eqn1}
G & = & \langle  x_1,\; x_2, \;x_3, \;x_4  \mid x_1^{p^n} = x_2^{p^2} = x_3^{p^2} = x_4^{p} = 1,
 \; [x_1, x_2] = x_2^p,\\
& &  [x_1, x_3] = x_3^p,\;  [x_1, x_4] = x_3^p,\; [x_2, x_3] = x_1^{p^{n-1}},\; [x_2, x_4] = x_2^p,\nonumber\\
& & [x_3, x_4] = 1 \rangle. \nonumber
\end{eqnarray}

Throughout this section, $G$ always denotes the group defined in \eqref{eqn1}.  

\begin{lemma}\label{lemma1}
 The group $G$ is a regular $p$-group of nilpotency class $2$, the exponent of $G$ is $p^n$ and $\Z(G) = \Phi(G)$.
\end{lemma}
\begin{proof}
By using the given relations, it is easy to show that the commutators $[x_i, x_j] \in \Z(G)$, where $1 \le i < j \le 4$. Thus for any elements $g_1, g_2 \in G$,  $[g_1, g_2]$ can be expressed as a product of powers of $[x_i, x_j]$, $1 \le i < j \le 4$. This shows that $\gamma_2(G) \subseteq \Z(G)$ and therefore the nilpotency class of $G$ is $2$. Since $p$ is odd, it follows that $G$ is regular. From the presentation of $G$ and the fact that $\gamma_2(G)$ is elementary abelian it follows that the exponent of $G$ is $p^n$. Since $G^p = \Z(G)$, it follows that $\Phi(G) = \gamma_2(G)G^p = \Z(G)$.  \hfill $\Box$

\end{proof}

The following two results are well known.

\begin{lemma}\label{lemma1a}
 Let $A$, $B$ and $C$ be finite abelian groups. Then $\Hom(A \times B, \;C) \cong \Hom(A, \; C) \times \Hom(B, \;C)$ and $\Hom(A, \;B \times C) \cong \Hom(A, \; B) \times \Hom(A, \;C)$.
\end{lemma}

\begin{lemma}\label{lemma1b}
 Let $\C_r$ and $\C_s$ be two cyclic groups of order $r$ and $s$ respectively. Then $\Hom(\C_r, $ $\C_s) \cong \C_d$, where $d$ is the greatest common divisor of $r$ and $s$.
\end{lemma}

A group $H$ is said to be \emph{purely non-abelian} if it does not have a non-trivial abelian direct factor.
Now we calculate the order of $\Autcent(G)$.
\begin{lemma}\label{lemma2}
 Let $G$ be the group defined in \eqref{eqn1}. Then  $|\Autcent(G)|= p^{n+10}$.
\end{lemma}
\begin{proof}
Since $\Z(G) = \Phi(G)$, $G$ is purely non-abelian. Then by Theorem 1 of \cite{AY65}, there is one-to-one correspondence between $\Autcent(G)$ and $\Hom\big(G/\gamma_2(G),$ $\Z(G)\big)$. Notice that $G/\gamma_2(G) \cong \C_{p^{n-1}} \times \C_p \times \C_p \times \C_p$ and $\Z(G) \cong   \C_{p^{n-1}} \times \C_p \times \C_p$. Thus using Lemma \ref{lemma1a} and Lemma \ref{lemma1b}, we have
\begin{eqnarray*}
|\Hom\big(G/\gamma_2(G),\; \Z(G)\big)| &\cong& |\Hom\big(\C_{p^{n-1}} \times \C_p \times \C_p \times \C_p,\; \C_{p^{n-1}} \times \C_p \times \C_p\big)|\\
&\cong& p^{n+1}p^3p^3p^3 = p^{n+10}.
\end{eqnarray*}
Hence $|\Autcent(G)| = p^{n+10}$. \hfill $\Box$

\end{proof}

Let
$e_{x_i} = x_1^{a_{i1}} x_2^{a_{i2}}x_3^{a_{i3}}x_4^{a_{i4}} =$  {\tiny$\prod$}$_{j=1}^4 x_j^{a_{ij}}$, where $x_i \in G$ and $a_{ij}$ are non-negative integers for
 $1 \le i, j \le 4$. Since the nilpotency class of $G$ is $2$, we have  
\begin{equation}\label{0a}
[x_k, e_{x_i}]  =  [x_k,  \text{\tiny$\prod$}_{j=1}^4 x_j^{a_{ij}}] = \text{\tiny$\prod$}_{j=1}^4 [x_k, x_j^{a_{ij}}] = \text{\tiny$\prod$}_{j=1}^4 [x_k, x_j]^{a_{ij}}
\end{equation}
and 
\begin{eqnarray}\label{0b}
 [e_{x_k}, e_{x_i}]  &=&  [\text{\tiny$\prod$}_{l=1}^4 x_l^{a_{kl}},  \text{\tiny$\prod$}_{j=1}^4 x_j^{a_{ij}}] = \text{\tiny$\prod$}_{j=1}^4 \text{\tiny$\prod$}_{l=1}^4 [x_l^{a_{kl}}, x_j^{a_{ij}}]\\
& = & \text{\tiny$\prod$}_{j=1}^4 \text{\tiny$\prod$}_{l=1}^4 [x_l, x_j]^{a_{kl} a_{ij}}.\nonumber
\end{eqnarray}

Equations \eqref{0a} and \eqref{0b} will be used for our calculations without any further reference. 

Let $\alpha$ be an automorphism of $G$. Since the nilpotency class of $G$ is $2$ and $\gamma_2(G)$ is generated by $x_1^{p^{n-1}}$, $x_2^p$, $x_3^p$, we can write $\alpha(x_i) = x_i e_{x_i} = x_i ${\tiny$\prod$}$_{j=1}^4 x_j^{a_{ij}}$ for some non-negative integers $a_{ij}$ for $1 \le i, j \le 4$.

\begin{prop}\label{prop1}
Let $G$ be the group defined in \eqref{eqn1} and $\alpha$ be an automorphism of $G$ such that $\alpha(x_i) = x_i e_{x_i} = x_i ${\tiny$\prod$}$_{j=1}^4 x_j^{a_{ij}}$, where $a_{ij}$ are some non-negative integers for $1 \le i, j \le 4$. Then the following equations hold:
\begin{eqnarray}
& &  a_{4j} \equiv 0 \mod{ p}, \;  \;1 \le j \le 3,\label{e1}\\
& & a_{21} \equiv 0 \mod{ p},\label{e2}\\
& & a_{11} + a_{11}a_{22} - a_{14} + a_{12} a_{24}  - a_{14}a_{22}\equiv 0 \mod{ p},\label{e3}\\
& & a_{11} a_{23} + a_{24} + a_{11} a_{24} \equiv 0 \mod{ p},\label{e4}\\
& & a_{31} \equiv 0 \mod{ p},\label{e5}\\
& & a_{11} + a_{11}a_{33} + a_{34} + a_{11} a_{34} \equiv 0 \mod{ p},\label{e7}\\
& & a_{12} -a_{32} + a_{12} a_{44}  \equiv 0 \mod{ p},\label{e8}\\
& & a_{44} -a_{33} + a_{11} + a_{11} a_{44}  \equiv 0 \mod{ p},\label{e9}\\
& & a_{33} + a_{22} + a_{22} a_{33} - a_{23} a_{32} -a_{11} \equiv 0 \mod{p},\label{e10}\\
& & a_{34} + a_{22} a_{34} - a_{24} a_{32} \equiv 0 \mod{p},\label{e11}\\
& & a_{44}+ a_{22} a_{44} \equiv 0 \mod{p},\label{e12}\\
& & a_{23} \equiv 0 \mod{p},\label{e13}\\
& & a_{32} +  a_{32} a_{44} \equiv 0 \mod{p}.\label{e14}
\end{eqnarray}
\end{prop}
\begin{proof}
Let $\alpha$ be the automorphism of $G$ such that $\alpha(x_i) = x_i e_{x_i}$, $1 \le i \le 4$. 
Since $G$ is regular, $G_p = \gen{x_1^{p^{n-1}}, x_2^p, x_3^p, x_4}$. We know that $G_p$ is characteristic, therefore $\alpha(x_4) = x_4 e_{x_4} \in G_p$. This shows that $e_{x_4} \in G_p$. Thus $a_{4j} \equiv 0 \mod{p}$ for $1 \le j \le 3$. This proves that equation \eqref{e1} holds.

We prove equations \eqref{e2} - \eqref{e4} by comparing the powers of $x_i$'s in $\alpha([x_1, x_2]) = \alpha(x_2^p)$.
\begin{eqnarray*}
 \alpha([x_1, x_2]) & = & [\alpha(x_1), \alpha(x_2)] = [x_1 e_{x_1}, x_2 e_{x_2}] \\
&=& [x_1, x_2] [x_1, e_{x_2}] [e_{x_1}, x_2] [e_{x_1}, e_{x_2}]\\
& = & [x_1, x_2] \text{\tiny$\prod$}_{j=1}^4 [x_1, x_j]^{a_{2j}} \text{\tiny$\prod$}_{j=1}^4 [x_2, x_j]^{-a_{1j}}  
\text{\tiny$\prod$}_{j=1}^4 \text{\tiny$\prod$}_{l=1}^4 [x_l, x_j]^{a_{1l} a_{2j}}\\
& = & [x_1, x_2]^{1 + a_{22} + a_{11} + a_{11}a_{22} - a_{12}a_{21}}
       [x_1, x_3]^{a_{23} + a_{11} a_{23} - a_{13} a_{21}}\\
& &    [x_1, x_4]^{a_{24} + a_{11} a_{24} - a_{14} a_{21}}
       [x_2, x_3]^{-a_{13} + a_{12} a_{23} - a_{13} a_{22}}\\
 & &   [x_2, x_4]^{-a_{14} + a_{12} a_{24} - a_{14} a_{22}}
       [x_3, x_4]^{a_{13} a_{24} - a_{14} a_{23}}\\
& = &  x_1^{p^{n-1}(-a_{13} + a_{12} a_{23} - a_{13} a_{22})}\\
& &  x_2^{p(1 + a_{22} + a_{11} + a_{11}a_{22} - a_{12}a_{21} - a_{14} + a_{12} a_{24} - a_{14} a_{22})}\\
& & x_3^{p(a_{23} + a_{11} a_{23} - a_{13} a_{21} + a_{24} + a_{11} a_{24} - a_{14} a_{21})}.
\end{eqnarray*}

On the other hand 
\[\alpha([x_1, x_2]) = \alpha(x_2^p) = x_2^p x_1^{pa_{21}} x_2^{pa_{22}}x_3^{pa_{23}}x_4^{pa_{24}} = 
x_1^{pa_{21}} x_2^{pa_{22} + p}x_3^{pa_{23}}.\]

Comparing the powers of $x_1$, we get $a_{21} \equiv 0 \mod{ p}$.
Using this fact and comparing the powers of $x_2$ and $x_3$, we get
\begin{eqnarray*}
 & &a_{11} + a_{11}a_{22} - a_{14} + a_{12} a_{24} -a_{14}a_{22} \equiv 0 \mod{ p},\\
& &  a_{11} a_{23} + a_{24} + a_{11} a_{24}\equiv 0 \mod{ p}.
\end{eqnarray*}
Hence equations \eqref{e2} - \eqref{e4} hold.

Equations \eqref{e5} and \eqref{e7} are proved in the same way by comparing the powers of $x_1$ and $x_3$ in $\alpha([x_1, x_3]) = \alpha(x_3^p)$. 
Equations \eqref{e8} and \eqref{e9} are proved by comparing the powers of $x_2$ and $x_3$ in $\alpha([x_1, x_4]) = \alpha(x_3^p)$ and using \eqref{e1}.
Equations \eqref{e10} and \eqref{e11} are proved by comparing the powers of $x_1$ and $x_2$ in $\alpha([x_2, x_3]) = \alpha(x_1^{p^{n-1}})$ and using \eqref{e2} and \eqref{e5}.
Equations \eqref{e12} and \eqref{e13} are proved by comparing the powers of $x_2$ and $x_3$ in $\alpha([x_2, x_4]) = \alpha(x_2^p)$ and using \eqref{e1} and \eqref{e2}.
Finally, equation \eqref{e14} is proved by comparing the powers of $x_2$ in $\alpha([x_3, x_4]) = 1$ and using \eqref{e1}.
\hfill $\Box$

\end{proof}

\begin{thm}\label{thm1}
 Let $G$ be the group defined in \eqref{eqn1}. Then all automorphisms of $G$ are central.
\end{thm}
\begin{proof}
Let $\alpha$ be an automorphism of $G$ such that $\alpha(x_i) = x_i e_{x_i}$,  where $e_{x_i}  = \text{\tiny$\prod$}_{j=1}^4 x_j^{a_{ij}}$,  $a_{ij}$ are non-negative integers for $1 \le i, j \le 4$. To complete the proof, it is sufficient to show that $e_{x_i} \in \Z(G)$ for $1 \le i \le 4$. Since $\Z(G) = \Phi(G)$, we shall show that $a_{ij} \equiv 0 \mod{p}$ for $1 \le i, j \le 4$.

We start by showing that $a_{44}+1$ is not divisible by $p$. For, if  $a_{44}+1$ is divisible by $p$, then it follows from equation \eqref{e1} that $\alpha(x_4) \in \Z(G)$, which is not possible.  Using this fact and equation \eqref{e14}, we get $a_{32} \equiv 0 \mod{p}$. Thus equation \eqref{e8} gives  $a_{12} \equiv 0 \mod{p}$. We claim that $p$ does not divide $a_{11}+1$. For, if $p$ divides $a_{11}+1$, then the order of $\alpha (x_1)$ is at most $p^{n-1}$, which is not possible. Using \eqref{e13} and the fact that $1+a_{11} \not\equiv 0 \mod{p}$, equation \eqref{e4} implies $a_{24} \equiv 0 \mod{p}$. Since $a_{2i} \equiv 0 \mod{p}$, $i = 1, 3, 4$, it follows that $\alpha(x_2) \equiv x_2^{1+a_{22}} \mod{\Z(G)}$. Hence $1+a_{22} \not\equiv 0 \mod{p}$.

Since $a_{32} \equiv 0 \mod{p}$ and $a_{22}+1 \not\equiv 0 \mod{p}$,  from equations \eqref{e11} and \eqref{e12}, we have $a_{34}$ and $a_{44}$ are both congruent to $0 \mod{p}$.  Here we claim that $a_{33}+1$ is not divisible by $p$. For, otherwise $\alpha(x_3) \in \Z(G)$, which is not possible. Now using the fact that $a_{34} \equiv 0 \mod{p}$, it follows from equation \eqref{e7} that $a_{11} \equiv 0 \mod{p}$. Since $a_{44} \equiv 0 \mod{p}$, from equation \eqref{e9} we have $a_{33} \equiv a_{11} \mod{p}$. Hence $a_{33} \equiv 0 \mod{p}$. Since $a_{11}$, $a_{12}$ are congruent to $0 \mod{p}$ and $1+ a_{22} \not\equiv 0 \mod{p}$, it follows from equation  \eqref{e3} that $a_{14} \equiv 0 \mod{p}$. Putting $a_{11}$, $a_{23}$ and $a_{33}$ equal to $0 \mod{p}$ in equation \eqref{e10}, we get $a_{22} \equiv 0 \mod{p}$. 

It only remains to show that $a_{13} \equiv 0 \mod{p}$. Using above information, notice that $e_{x_2}, e_{x_4} \in \Z(G)$. Thus 
\[x_2^p = [x_2, x_4] = \alpha([x_2, x_4]) = \alpha(x_2^p) = x_1^{pa_{21}}x_2^{p(a_{22} +1)}x_3^{pa_{23}}.\] 
This implies that $x_1^{pa_{21}} = 1$. Since $e_{x_2}, x_1^{a_{11}}, x_2^{a_{12}}, x_4^{a_{14}} \in \Z(G)$, we get 
\[\alpha([x_1, x_2]) = [x_1, x_2][x_3, x_2]^{a_{13}} = x_2^p x_1^{-p^{n-1}a_{13}}.\]
 This gives
\[x_2^p x_1^{-p^{n-1}a_{13}} = \alpha([x_1, x_2]) = \alpha(x_2^p) = x_1^{pa_{21}}x_2^{p(a_{22} +1)}x_3^{pa_{23}}.\] 
This, by comparing the powers of $x_1$ and using the fact that $x_1^{pa_{21}} = 1$, implies $x_1^{-p^{n-1}a_{13}} = 1$, which in turn implies that $a_{13} \equiv 0 \mod{p}$. Hence $a_{ij} \equiv 0 \mod{p}$ for all $1 \le i, j \le 4$. Since $\alpha$ is an arbitrary automorphism of $G$, this completes the proof of the theorem. 
\hfill $\Box$

 \end{proof}

Let $A$ be an abelian $p$-group and $a \in A$. For a positive integer $n$, $p^n$  is said to be the \emph{height} of $a$ in $A$, denoted by $\het(a)$, if $a \in A^{p^n}$ but
$a \not\in  A^{p^{n+1}}$. Let $H$ be a  $p$-group of class $2$. We denote the exponents of $\Z(H)$, $\gamma_2(H)$, $H/\gamma_2(H)$ by $p^a$, $p^b$, $p^c$ respectively and $d = \text{min}(a, c)$. We define $R := \{z \in \Z(H) \mid \; |z| \le p^d\}$ and $K := \{x \in H \mid \het(x\gamma_2(H)) \ge p^b\}$. Notice that $K = H^{p^b}\gamma_2(H)$. 
To complete the proof of Theorem A we need the following result (in our notations) of J. E. Adney and T. Yen \cite[Theorem 4]{AY65}.
\begin{thm}\label{thm2}
Let $H$ be a  purely non-abelian $p$-group of class $2$, $p$ odd, and let $H/\gamma_2(H) = \Pi_{i = 1}^n\gen{x_i\gamma_2(H)}$. Then $\Autcent(H)$ is abelian if and only if

\text{(i)} $R = K$, and 

\text{(ii)} either $d = b$ or $d > b$ and $R/\gamma_2(H) = \gen{x_1^{p^b}\gamma_2(H)}$.
\end{thm}

Now we are in the position to complete the proof of Theorem A stated in the introduction.

\begin{proof}[Proof of Theorem A]
Let $G$ be the group defined in \eqref{eqn1}. By  Lemma \ref{lemma1}, we have  $|G| = p^{n+5}$ and the exponent of $G$ is $p^n$. By Theorem \ref{thm1}, we have $\Aut(G) = \Autcent(G)$. Now it follows from Lemma \ref{lemma2} that $|\Aut(G)| = p^{n+10}$. Thus to complete the proof of the theorem, it is sufficient to prove that $\Autcent(G)$ is an abelian group. Since $\Z(G) = \Phi(G)$ (by Lemma \ref{lemma1}), $G$ is purely non-abelian. The exponents of $\Z(G)$, $\gamma_2(G)$ and $G/\gamma_2(G)$ are $p^{n-1}$, $p$ and $p^{n-1}$ respectively. Here 
\[R =\{z \in \Z(G) \mid \; |z| \le p^{n-1}\} = \Z(G)\]
and 
\[K =  \{x \in G \mid \het(x\gamma_2(G)) \ge p\} = G^{p}\gamma_2(G) = \Z(G).\]
This shows that $R = K$. Also $R/ \gamma_2(G) = \Z(G)/ \gamma_2(G) = \gen{x_1^p\gamma_2(G)}$. Thus all the conditions of Theorem \ref{thm2} are now satisfied. Hence $\Autcent(G)$ is abelian. This completes the proof of the theorem. \hfill $\Box$

\end{proof}

\section{Groups $G$ such that $\Aut(G) = \Autcent(G)$ is non-abelian}

In this section we construct examples of finite $p$-groups $G$ such that $\Aut(G) = \Autcent(G)$ is non-abelian, where $p$ is an odd prime. 

Let $n$ be a natural number greater than $2$ and $p$ an odd prime. Define
\begin{eqnarray}\label{eqn2}
G & = & \langle  x_1,\; x_2, \;x_3, \;x_4\; |\; x_1^{p^n} = x_2^{p^3} = x_3^{p^2} = x_4^{p^2} = 1,
 \; [x_1, x_2] = x_2^{p^2},\\
& &  [x_1, x_3] = x_3^p,\;  [x_1, x_4] = x_4^p,\; [x_2, x_3] = x_1^{p^{n-1}},\; [x_2, x_4] = x_2^{p^2},\nonumber\\
& & [x_3, x_4] = x_4^p \rangle. \nonumber
\end{eqnarray}

This group $G$ is a regular $p$-group of nilpotency class $2$ having order $p^{n+7}$ and exponent $p^n$. Further $\Z(G) = \Phi(G)$ and therefore $G$ is purely non-abelian.

Let $\alpha$ be an automorphism of $G$. Since the nilpotency class of $G$ is $2$ and $\gamma_2(G)$ is generated by $x_1^{p^{n-1}}$, $x_2^{p^2}$, $x_3^p$, $x_4^p$, we can write $\alpha(x_i) = x_i e_{x_i} = x_i \prod_{j=1}^4 x_j^{a_{ij}}$ for some non-negative integers $a_{ij}$ for $1 \le i, j \le 4$.

\begin{prop}\label{prop2}
Let $G$ be the group defined in \eqref{eqn2} and $\alpha$ be an
automorphism of $G$ such that $\alpha(x_i) = x_i e_{x_i} = x_i${\tiny $\prod$}$_{j=1}^4
x_j^{a_{ij}}$, where $a_{ij}$ are some non-negative integers for $1 \le i,
j \le 4$. Then the following equations hold:
\begin{eqnarray}
&& a_{3i} \equiv 0 \mod{p}, \;\; a_{4i} \equiv 0 \mod{p}, \;\;\text{where}\;\; i = 1, 2, \label{f1}\\
&&  a_{43}  \equiv 0 \mod{p},\label{f2} \\
&&  1 + a_{44} \not\equiv 0 \mod{p}, \label{f3}\\
&& a_{33}  \equiv 0 \mod{p}, \label{f4}\\
&& a_{21} \equiv 0 \mod{p}, \label{f5}\\
&& a_{44} (1 + a_{22}) \equiv 0 \mod{p}, \label{f6}\\
&& a_{23} \equiv 0 \mod{p}, \label{f7}\\
&& a_{22} - a_{11} \equiv 0 \mod{p}, \label{f8}\\
&& a_{24} \equiv 0 \mod{ p}, \label{f9}\\
&& a_{11} \equiv 0 \mod{p}, \label{f10}\\
&& a_{13}a_{34} - a_{14}\equiv 0 \mod{p}, \label{f11}\\
&& a_{13} \equiv 0 \mod{p}. \label{f12}
\end{eqnarray}
\end{prop}
\begin{proof}
Let $\alpha$ be the automorphism of $G$ such that $\alpha(x_i) = x_i e_{x_i}$, $1 \le i \le 4$. 
Since $G$ is regular, $G_{p^2} = \gen{x_1^{p^{n-2}}, x_2^p, x_3, x_4}$. We know that $G_{p^2}$ is characteristic, therefore $\alpha(x_i) = x_i e_{x_i} \in G_{p^2}$, where $i = 3, 4$. This shows that $e_{x_i} \in G_{p^2}$ for $i = 3, 4$. Thus $a_{ij} \equiv 0 \mod{p}$ for $i = 3,4$ and $j = 1, 2$. This proves that equation \eqref{f1} holds.

Equation \eqref{f2} is proved by comparing the powers of $x_3$ in $\alpha([x_3, x_4]) = \alpha(x_4^p)$ and using \eqref{f1}. Since $a_{4i} \equiv 0 \mod{p}$, $1 \le i \le 3$, it follows that $x_1^{a_{41}}x_2^{a_{42}}x_3^{a_{43}} \in \Z(G)$. Suppose $1 + a_{44} \equiv 0 \mod{p}$. Then $\alpha(x_4) = x_4 e_{x_4} =
x_4^{1 + a_{44}}x_1^{a_{41}}x_2^{a_{42}}x_3^{a_{43}} \in \Z(G)$, which is not possible. This proves equation \eqref{f3}. Equation \eqref{f4} is proved     
by comparing the powers of $x_4$ in $\alpha([x_3, x_4]) = \alpha(x_4^p)$ and using \eqref{f1} - \eqref{f3}.

Equations \eqref{f5} - \eqref{f7} are proved by comparing the powers of $x_1$, $x_2$ and $x_4$ in $\alpha([x_2, x_4]) = \alpha(x_2^{p^2})$ and using the equations above. Equations \eqref{f8} and \eqref{f9} are proved by comparing the powers of $x_1$ and $x_4$ in $\alpha([x_2, x_3]) = \alpha(x_1^{p^{n-1}})$ and using the equations above. Equations \eqref{f10}, \eqref{f11} are proved by comparing the powers of $x_3$, $x_4$ in $\alpha([x_1, x_3]) = \alpha(x_3^p)$ and using the equations above. Finally, equation \eqref{f12} is proved by comparing the powers of $x_4$ in $\alpha([x_1, x_4]) = \alpha(x_4^p)$ and using the equations above. \hfill $\Box$

\end{proof}

\begin{proof}[Proof of Theorem B]
 Let $G$ be the group defined in \eqref{eqn2}. We know that $G$ is a purely non-abelian group of order $p^{n+7}$ and  exponent $p^n$.
Let $\alpha$ be an automorphism of $G$. Then $\alpha(x_i) = x_i${\tiny$\prod$}$_{j=1}^4x_j^{a_{ij}}$. 
To prove $\Aut(G) = \Autcent(G)$, it is sufficient to show that $e_{x_i} \in \Z(G)$ for $1 \le i \le 4$. Since $\Z(G) = \Phi(G)$, we only need to show that $a_{ij} \equiv 0 \mod{p}$ for $1 \le i, j \le 4$. 

From Proposition \ref{prop2} it is obvious that $a_{ij}$ is congruent to $0$ modulo $p$ for $1 \le i, j \le 4$, except for $a_{12}$ and $a_{34}$. Since $e_{x_4} \in \Z(G)$ and $x_1^{a_{31}}x_2^{a_{32}}x_3^{a_{33}} \in \Z(G)$, we have $\alpha([x_3, x_4]) = [x_3 x_4^{a_{34}}, x_4] = [x_3, x_4] = x_4^p$. Furthermore, $\alpha([x_1, x_4]) = [x_1x_2^{a_{12}}, x_4] = [x_1, x_4][x_2, x_4]^{a_{12}} = x_4^p x_2^{p^{2}a_{12}}$, since $x_1^{a_{11}}x_3^{a_{13}}x_4^{a_{14}} \in \Z(G)$. From the presentation of the group, we have $[x_3, x_4] = [x_1, x_4]$. Thus  $x_4^p = \alpha([x_3, x_4]) = \alpha([x_1, x_4]) = x_4^p x_2^{p^{2}a_{12}}$. Hence $x_2^{p^{2}a_{12}} = 1$, which proves that $a_{12} \equiv 0 \mod{p}$.  Notice that $e_{x_1}$ and $e_{x_2}$ lie in $\Z(G)$ and $(e_{x_1})^{p^{n-1}} = 1$. Using this and comparing the powers of $x_2$ in $\alpha([x_2, x_3]) = \alpha(x_1^{p^{n-1}})$, we get $a_{34} \equiv 0 \mod{p}$. This proves that $\Aut(G) = \Autcent(G)$.

Now we proceed to show that $\Autcent(G)$ is non-abelian.
The exponents of $\Z(G)$, $\gamma_2(G)$ and $G/\gamma_2(G)$ are $p^{n-1}$, $p$ and $p^{n-1}$ respectively. Here 
\[R =\{z \in \Z(G) \mid \; |z| \le p^{n-1}\} = \Z(G)\]
and 
\[K =  \{x \in G \mid \het(x\gamma_2(G)) \ge p\} = G^{p}\gamma_2(G) = \Z(G).\]
This shows that $R = K$. Now $R/ \gamma_2(G) = \Z(G)/ \gamma_2(G) = \gen{x_1^p\gamma_2(G)} \times \gen{x_2^p\gamma_2(G)}$. This shows that condition $\text{(ii)}$ of Theorem \ref{thm2} is not satisfied. Hence $\Autcent(G)$ is non-abelian. This completes the proof of the theorem. \hfill $\Box$

\end{proof}

\noindent{\bf Acknowledgements.} We thank Derek Holt and  Bettina Eick for their help in using GAP for calculating automorphisms of finite $p$-groups, and cluster computing facility at our institute for running GAP on it. We thank the referee for giving useful comments and suggestions to make the paper more readable. Reference \cite{pH95} is brought to our knowledge by the referee. We thank M. J. Curran for bringing references \cite{mC87, HL74, rS82} to our knowledge and giving some useful suggestions.

\end{document}